\documentclass[12pt, reqno]{amsart}       
\usepackage{graphicx}

\usepackage{amsmath}

\usepackage{amssymb}
\usepackage{pifont}
\newtheorem{thm}{Theorem}

\theoremstyle{definition}

\newcommand{\vertiii}[1]{{\left\vert\kern-0.25ex\left\vert\kern-0.25ex\left\vert
		#1 \right\vert\kern-0.25ex\right\vert\kern-0.25ex\right\vert}}

\def \lim   {\text {\rm lim}}

\begin{document}
	
	\title[]{A proof of Sylvester's theorem}
	\author{Saptak Bhattacharya}
	\address{Indian Statistical Institute\\
		New Delhi 110016\\
		India}
	\email{saptak21r@isid.ac.in}
	\begin{abstract}We give a new elementary proof of existence and uniqueness of a solution to the Sylvester equation $AX-XB=Y$\end{abstract}
	\maketitle
	\section{Introduction}
	Given finite dimensional Hilbert spaces $\mathcal{H}$, $\mathcal{K}$ and operators $A\in\mathcal{L}(\mathcal{K})$, $B\in\mathcal{L}(\mathcal{H})$ and $Y\in\mathcal{L}(\mathcal{H},{K})$ the Sylvester equation asks for solutions $X\in\mathcal{L}(\mathcal{H},{K})$ to \[AX-XB=Y\label{e1}\tag{1}\]
	A particular case of interest is the Lyapunov equation \[A^*X+XA=Y\label{e2}\tag{2}\] which arises in stability theory (see \cite{par}). Equation \eqref{e1} was first studied by Sylvester in \cite{syl}, who showed that it has a unique solution if $\sigma(A)\cap\sigma(B)=\emptyset$. This was generalized to infinite dimensions by Rosenblum in \cite{ros}.
	\medskip
	
	The purpose of this note is to give a short proof of Sylvester's theorem using elementary block matrix arguments. Other different proofs are given in \cite{bhr, nh, dym}.  A thorough survey on equation \eqref{e1} can be found in \cite{bhr}.
	\section{Main result}
	\begin{thm}\label{t1}Let $\mathcal{H}$ and $\mathcal{K}$ be finite dimensional Hilbert spaces and let $A\in \mathcal{L}(\mathcal{K})$ and $B\in\mathcal{L}(\mathcal{H})$ with $\sigma(A)\cap\sigma(B)=\emptyset$. Then for every $Y\in \mathcal{L}(\mathcal{H},\mathcal{K})$ there exists a unique $X\in\mathcal{L}(\mathcal{H},\mathcal{K})$ such that $AX-XB=Y$.\end{thm}
	\begin{proof}Consider the map $\Phi:\mathcal{L}(\mathcal{H},\mathcal{K})\to\mathcal{L}(\mathcal{H},\mathcal{K})$ given by $\Phi(X)=AX-XB$. It suffices to show that $\Phi$ is injective. If ker $\Phi$ contains an invertible $X$, we have $$X^{-1}AX=B$$ implying $\sigma(A)=\sigma(B)$, a contradiction. If not, we use a block matrix argument to reduce to this case. Let $X\in$ ker $\Phi$ such that $X\neq O$. Consider the direct sum decompositions $$\mathcal{H}=(\textrm{ker}\hspace{0.7mm}X)^{\perp}\oplus\textrm{ker}\hspace{0.7mm}X$$ and $$\mathcal{K}=\textrm{im}\hspace{0.7mm}X\oplus\textrm{ker}\hspace{0.7mm}X^{*}.$$ Note that $(\textrm{ker}\hspace{0.7mm}X)^{\perp}\neq\{0\}$.
	With respect to these decompositions, we have the block matrices $$X=\begin{pmatrix}Y & O\\O & O\end{pmatrix}$$ $$A=\begin{pmatrix}E & F\\G & H\end{pmatrix}$$and$$B=\begin{pmatrix}P & Q\\R & S\end{pmatrix}.$$ Observe that $Y$ is invertible. The condition $AX=XB$ now yields $$\begin{pmatrix}EY & O\\GY & O\end{pmatrix}=\begin{pmatrix}YP & YQ\\O & O\end{pmatrix}$$ implying \[EY=YP,\label{e3}\tag{3}\]$$G=O\hspace{2mm}\textrm{and}\hspace{2mm}Q=O.$$ Now $$A=\begin{pmatrix}E & F\\O & H\end{pmatrix}$$ and$$B=\begin{pmatrix}P & O\\R & S\end{pmatrix}.$$ Thus, $\sigma(E)\subset\sigma(A)$ and $\sigma(P)=\sigma(P^t)\subset\sigma(B^t)=\sigma(B)$ which implies $\sigma(E)\cap\sigma(P)=\emptyset$. From \eqref{e3}, we have a contradiction due to the invertibility of $Y$.
	\end{proof}
    
	\bibliographystyle{amsplain}
	
\end{document}